\documentclass[12pt]{article}
\usepackage[amsmath]{e-jc_modified}
\usepackage[dvipsnames]{xcolor}
\usepackage{dsfont}

\let\E=\EE
\let\epsilon=\varepsilon
\let\eps=\epsilon
\def\PP{\Pr}
\def\rm{\mathbb{RM}}

\DeclareMathOperator{\match}{match}
\DeclareMathOperator{\perm}{perm}

\AtBeginDocument{%
   \def\MR#1{}
}

\usepackage{graphicx}

\dateline{Jan 1, 2024}{Jan 2, 2024}{TBD}

\MSC{05C80, 05C30, 05C70}

\Copyright{The authors. Released under the CC BY-ND license (International 4.0).}

\title{Ordered unavoidable sub-structures in matchings and random matchings}

\author{Andrzej Dudek\thanks{Supported in part by Simons Foundation Grant \#522400.}\\
\small Department of Mathematics\\[-0.8ex]
\small Western Michigan University\\[-0.8ex]
\small Kalamazoo, MI, USA\\
\small\tt andrzej.dudek@wmich.edu\\
\and
Jaros\l aw Grytczuk\thanks{Supported in part by Narodowe Centrum Nauki, grant 2015/17/B/ST1/02660.}\\
\small Faculty of Mathematics and Information Science\\[-0.8ex]
\small Warsaw University of Technology\\[-0.8ex]
\small Warsaw, Poland\\
\small\tt j.grytczuk@mini.pw.edu.pl\\
\and
Andrzej Ruci\'nski\thanks{Supported in part by Narodowe Centrum Nauki, grant 2018/29/B/ST1/00426.}\\
\small Department of Discrete Mathematics\\[-0.8ex]
\small Adam Mickiewicz University\\[-0.8ex]
\small Pozna\'n, Poland\\
\small\tt rucinski@amu.edu.pl}

\author{Andrzej Dudek\authornote{1}
\and
Jaros{\l a}w Grytczuk\authornote{2}
\and
Andrzej Ruci\'nski\authornote{3}
}

\authortext{1}{Department of Mathematics, Western Michigan University, Kalamazoo, MI, USA
(\email{andrzej.dudek@wmich.edu}).}
\authortext{2}{Faculty of Mathematics and Information Science, Warsaw University of Technology, Warsaw, Poland
(\email{j.grytczuk@mini.pw.edu.pl}).}
\authortext{3}{Department of Discrete Mathematics, Adam Mickiewicz University, Pozna\'n, Poland,
(\email{rucinski@amu.edu.pl}).}

\begin{document}

\maketitle

\begin{abstract} An \emph{ordered matching of size $n$} is a graph on a linearly ordered vertex set~$V$, $|V|=2n$, consisting of $n$ pairwise disjoint edges. There are three different ordered matchings of size two on $V=\{1,2,3,4\}$: an \emph{alignment} $\{1,2\},\{3,4\}$,   a \emph{nesting} $\{1,4\},\{2,3\}$, and a \emph{crossing} $\{1,3\},\{2,4\}$. Accordingly, there are three basic homogeneous types of ordered matchings (with all pairs of edges arranged  in the same way) which we call, respectively, \emph{lines}, \emph{stacks}, and \emph{waves}.
	
We prove an Erd\H{o}s--Szekeres type result guaranteeing in every ordered matching of size~$n$ the presence of one of the three basic sub-structures of a given size. In particular, one of them must be of size at least $n^{1/3}$. We also investigate the size of each of the three sub-structures in a \emph{random} ordered matching. Additionally, the former result is generalized to $3$-uniform ordered matchings.

Another type of unavoidable patterns we study are \emph{twins}, that is,  pairs of order-isomorphic, disjoint sub-matchings. By relating to a similar problem for permutations, we prove that the maximum size of twins that occur in every ordered matching of size~$n$ is $O\left(n^{2/3}\right)$ and $\Omega\left(n^{3/5}\right)$. We conjecture that the upper bound is the correct order of magnitude  and  confirm it for almost all matchings. In fact, our results for twins are proved more generally for $r$-multiple twins, $r\ge2$. \footnote{An extended abstract of this paper appears in~\cite{DGR_LATIN}.}
\end{abstract}

\maketitle



\section{Introduction}
\subsection{Background}
A graph $G$ is said to be \emph{ordered} if its vertex set is linearly ordered.
Let $G$ and $H$ be two ordered graphs with vertex sets $V(G)=\{v_1,\ldots,v_m\}$ and $V(H)=\{w_1,\ldots, w_m\}$, and the respective linear orders $v_1<\cdots <v_m$ and $w_1<\cdots<w_m$, for some integer $m\ge1$. We say that $G$ and $H$ are \emph{order-isomorphic} if for all $1\le i<j\le m$, $v_iv_j\in E(G)$ if and only if $w_iw_j\in E(H)$. Note that every pair of order-isomorphic graphs is isomorphic, but not vice-versa. Also, if $G$ and $H$ are distinct graphs on the same linearly ordered vertex set $V$, then $G$ and $H$ are never order-isomorphic, and so all $2^{\binom{|V|}2}$ labeled graphs on $V$ are pairwise non-order-isomorphic. It shows that the notion of order-isomorphism makes sense only for pairs of graphs on distinct vertex sets.

One context in which order-isomorphism makes quite a difference is that of subgraph containment. If $G$ is an ordered graph, then any subgraph $G'$ of $G$ can be also treated as an ordered graph with the ordering of $V(G')$ inherited from the ordering of $V(G)$.
Given  two ordered graphs, (a ``large'' one) $G$ and (a ``small'' one) $H$, we say that a subgraph $G'\subset G$ is \emph{an ordered copy of $H$ in $G$} if $G'$ and $H$ are order-isomorphic. We will sometimes denote this fact by writing $G'\preceq G$.

All kinds of questions concerning subgraphs in unordered graphs can be posed for ordered graphs as well (see, e.g., \cite{Tardos} and \cite{BGT}). For example, in \cite{BCKK} and \cite{ConlonFoxSudakov} the authors studied Tur\'an and Ramsey type problems for ordered graphs. In particular, they showed independently that there exists an ordered matching on $n$ vertices for which the (ordered) Ramsey number is super-polynomial in~$n$, a sharp contrast with the linearity of the Ramsey number for ordinary (i.e.,~unordered) matchings. This shows that it makes sense to study even such seemingly simple structures as ordered matchings. In fact, Jel\'inek  \cite{VJ} counted the number of matchings avoiding (i.e.,~not containing) a given small ordered matching.

\subsection{Topics and organization}
In this paper we focus exclusively on  \emph{ordered matchings}, that is, ordered graphs which consist of vertex-disjoint edges (and have no isolated vertices).
For example, in Figure \ref{Matchings1}, we depict two ordered matchings, $M=\{\{1,3\},\{2,4\},\{5,6\}\}$ and $N=\{\{1,5\},\{2,3\},\{4,6\}\}$ on vertex set $\{1,2,3,4,5,6\}$ with the natural linear ordering.
\begin{figure}[ht]
	
	\begin{center}
		
		\resizebox{13cm}{!}{
			
			\includegraphics{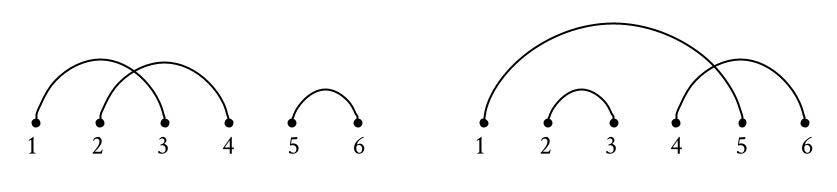}
			
		}
		
		\caption{Exemplary matchings $M$ and $N$.}
		\label{Matchings1}
	\end{center}
	
\end{figure}
Unlike in \cite{VJ}, we study what sub-structures are \emph{unavoidable} in ordered matchings. A frequent theme in both fields, the theory of ordered graphs as well as enumerative combinatorics, are unavoidable sub-structures, that is, \emph{patterns} that appear in every member of a prescribed family of structures. A flagship example providing everlasting inspiration is the famous theorem of Erd\H {o}s and Szekeres \cite{ErdosSzekeres} on monotone subsequences (see \cite{BKP,BucicSudakovTran,BukhMatousek,EliasMatousek,FoxPachSudakovSuk,MoshkovitzShapira,SzaboTardos} for some recent extensions and generalizations). In its diagonal form it states that any sequence $x_1,x_2,\ldots, x_n$ of distinct real numbers contains an increasing or decreasing subsequence of length at least~$\sqrt n$.

 And, indeed, our first goal is to prove its analog for ordered matchings. The reason why the original  Erd\H os--Szekeres Theorem lists only two types of subsequences is, obviously, that for any two elements $x_i$ and $x_j$ with $i<j$ there are just two possible relations: $x_i< x_j$ or $x_i> x_j$. For matchings, however,  for every two  edges $\{x,y\}$ and $\{u,w\}$ with $x<y$, $u<w$, and $x<u$,  there are three possibilities: $y<u$ , $w<y$, or $u<y<w$ (see Figure~\ref{Matchings6}). In other words, every two edges form either \emph{an alignment}, \emph{a nesting}, or \emph{a crossing} (the first term introduced by Kasraoui and Zeng in \cite{KZ}, the last two terms coined in by Stanley \cite{Stanley}). These three possibilities give rise, respectively,
 to three ``unavoidable'' ordered sub-matchings (\emph{lines}, \emph{stacks}, and \emph{waves}) which play an analogous role to the monotone subsequences in the classical Erd\H os--Szekeres Theorem. (In \cite{Stanley}, stacks and waves consisting of $k$ edges were called, respectively, \emph{$k$-nestings} and \emph{$k$-crossings}.)

 \begin{figure}[ht]
 	
 	\begin{center}
 		
 		\resizebox{13cm}{!}{
 			
 			\includegraphics{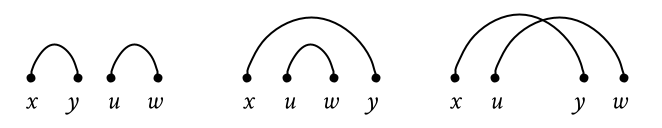}
 			
 		}
 		
 		\caption{An alignment, a nesting, and a crossing of a pair of edges.}
 		\label{Matchings6}
 	\end{center}
 	
 \end{figure}
  Informally, lines, stacks, and waves are defined by demanding that every pair of edges in a sub-matching  forms, respectively, an alignment, a nesting, or a crossing (see Figure \ref{Matchings3}).
  In Subsection \ref{arb_match} we  show, in particular, that every ordered matching of size $n$ contains one of these structures of size at least  $n^{1/3}$. This special case was also proved by Huynh, Joos and Wollan
(see Lemma~25 in~\cite{HJW2019}). In the remainder of Section \ref{unav}, we first extend this result to 3-uniform ordered matchings and then study the size of the largest lines, stacks, and waves in random matchings.

 Our second goal is to estimate the size of the largest (ordered) twins in ordered matchings. The problem of twins has been widely studied for other combinatorial structures, including words, permutations, and graphs (see, e.g., \cite{APP,LLSud}).
For an integer $r\ge2$, we say that $r$ edge-disjoint (ordered) subgraphs $G_1,G_2,\dots, G_r$ of an (ordered) graph $G$ form \emph{(ordered) twins in $G$} if they are pairwise (order-)isomorphic. The size of the (ordered) twins is defined as $|E(G_1)|=\cdots=|E(G_r)|$. For ordinary matchings, the notion of $r$-twins becomes trivial, as every matching of size $n$ contains twins of size $\lfloor n/r\rfloor$ -- just split the matching into $r$ as equal as possible parts. But for ordered matchings the problem becomes interesting. The above mentioned analog of the Erd\H os--Szekeres Theorem immediately yields (again by splitting into $r$ equal parts) ordered $r$-twins of length $\lfloor n^{1/3}/r\rfloor$. We provide much better estimates on the size of largest $r$-twins in ordered matchings (Subsection \ref{twins_am}) and random matchings (Subsection \ref{twins_rm}) which, not so surprisingly,  are of the same order of magnitude as those for $r$-twins in permutations (see \cite{BukhR,DGR}).

\subsection{Random matchings}\label{rm}
 As indicated above, we examine both questions, of unavoidable sub-matchings and of twins, also for \emph{random} matchings.  A random (ordered) matching $\rm_{n}$ is selected uniformly at random from all
\[
\alpha_n:=\frac{(2n)!}{n!2^n}
\]
ordered matchings on vertex set $[2n]$. Among other results, we show that with probability tending to 1, as $n\to\infty$, or \emph{asymptotically almost surely} (a.a.s.), there are in $\rm_{n}$ lines, stacks, \emph{and} waves of size, roughly, $\sqrt n$, as well as (ordered) twins of size $\Theta(n^{2/3})$.

There are two other ways of generating $\rm_{n}$ which we are going to utilize in the proofs. Besides the above defined \emph{uniform} scheme, we define the \emph{online} scheme as follows. For an arbitrary ordering of the vertices $u_1,\dots,u_{2n}$ one selects uniformly at random a match, say $u_{j_1}$, for $u_1$ (in $2n-1$ ways), then, after crossing out $u_1$ and $u_{j_1}$ from the list, one selects uniformly at random a match for the first  uncrossed vertex (in $2n-3$ ways), and so on. Note that the total number of ways to select a matching this way is $(2n-1)(2n-3)\cdot\ldots\cdot 3\cdot 1 = (2n-1)!!$ which equals $\alpha_n$. A third equivalent way to generate $\rm_{n}$ is particularly convenient when  one intends to apply concentration inequalities available for random permutations. The \emph{permutation based} scheme boils down to just generating a random permutation $\Pi:=\Pi_n$ of $[2n]$ and ``chopping it off'' into a matching $\{\Pi(1)\Pi(2),\;\Pi(3)\Pi(4),\dots,\Pi(2n-1)\Pi(2n)\}$. Note that this way each matching corresponds, as it should, to exactly $n!2^n$ permutations.
We will stick mostly to the uniform scheme, applying the other two only occasionally.

 \subsection{Gauss codes}

A convenient representation of ordered matchings can be obtained in terms of \emph{double occurrence words} over an $n$-letter alphabet, in which every letter occurs exactly twice as the label of the ends of the corresponding edge in the matching. For instance, our two exemplary matchings can be written as $M=ABABCC$ and $N=ABBCAC$ (see Figure~\ref{Matchings2}).
In fact, this is a special case of an elementary combinatorial bijection between ordered partitions of a set and permutations with repetitions.
 A minor nuisance here is that ordered matchings correspond to unordered partitions, so every permutation of the letters yields the same matching.
 Nevertheless, we will sometimes  use this representation  to better illustrate some notions and ideas.

\begin{figure}[ht]
	
	\begin{center}
		
		\resizebox{13cm}{!}{
			
			\includegraphics{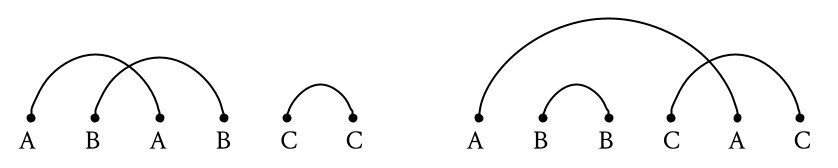}
			
		}
		
		\caption{Exemplary matchings $M=ABABCC$ and $N=ABBCAC$.}
		\label{Matchings2}
	\end{center}
	
\end{figure}
Interestingly, this type of words was introduced and studied already by Gauss \cite{Gauss} as a way of encoding closed self-intersecting curves on the plane (with points of multiplicity at most two). Indeed, denoting the self-crossing points by letters and traversing the curve in a fixed direction gives a (cyclic) word in which every letter occurs exactly twice (by the multiplicity assumption) (see Figure \ref{Matchings4}). A general problem studied by Gauss was to characterize those words that correspond to such curves. He found himself a necessary condition, but a full solution (in terms of some quite involved constraints on an auxiliary  graph of crossing pairs) was obtained much later (see \cite{ShtyllaTZ} for a brief history and further references).

It is, perhaps, also worthwhile to mention that ordered matchings constitute a special case of structures, known as \emph{Puttenham diagrams}, that found an early application in the theory of poetry (see \cite{Puttenham,PuttenhamEdited}). A basic idea is simple: a rhyme scheme of a poem can be encoded by a word in which same letters correspond to rhyming verses. Of particular interest here are \emph{planar rhyme schemes} which are nothing else but ordered matchings without crossings, or more generally, \emph{noncrossing partitions} (see \cite{StanleyCatalan}).

\begin{figure}[ht]
	
	\begin{center}
		
		\resizebox{10cm}{!}{
			
			\includegraphics{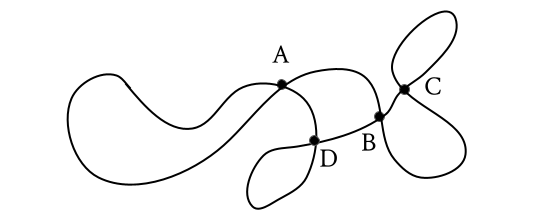}
			
		}
		
		\caption{A curve with Gauss code $ABCCBDDA$.}
		\label{Matchings4}
	\end{center}
	
\end{figure}

\section{Unavoidable sub-matchings}\label{unav}

Let us start with  formal definitions. Let $M$ be an ordered matching on the vertex set $[2n]$, with edges denoted as $e_i=\{a_i,b_i\}$ so that $a_i<b_i$, for all $i=1,2,\ldots,n$, and $a_1<\cdots<a_n$. We say that an edge $e_i$ is \emph{to the left of} $e_j$ and write $e_i<e_j$ if $a_i<a_j$. That is, in ordering the edges of a matching we ignore the positions of the right endpoints.

We  now define the three basic types of ordered matchings:
\begin{itemize}
	\item Line: $a_1<b_1<a_2<b_2<\cdots<a_n<b_n$,
	\item Stack: $a_1<a_2<\cdots<a_n<b_n<b_{n-1}<\cdots<b_1$,
	\item Wave: $a_1<a_2<\cdots<a_n<b_1<b_2<\cdots<b_n$.
\end{itemize}
Assigning letter $A_i$ to edge $\{a_i,b_i\}$, their corresponding double occurrence  words look as follows:
\begin{itemize}
	\item Line: $\quad A_1A_1A_2A_2\cdots A_nA_n$,
	\item Stack: $A_1A_2\cdots A_{n-1}A_nA_nA_{n-1}\cdots A_1$.
	\item Wave: $A_1A_2\cdots A_nA_1A_2\cdots A_n$.
\end{itemize}
Each of these three types of ordered matchings can be equivalently characterized as follows. Let us consider all possible ordered matchings with just two edges. In the double occurrence word notation these are $AABB$ (an \emph{alignment}), $ABBA$ (a \emph{nesting}), and $ABAB$ (a \emph{crossing}). Now a line, a stack, and a wave is an ordered matching in which \emph{every} pair of edges forms, respectively, an alignment, a nesting, and  a crossing (see Figure \ref{Matchings3}).

Note that alignments, crossings and nestings are just special instances (the smallest non-trivial) of, resp., lines, stacks, and waves, and throughout  we will use these names interchangeably.

\begin{figure}[ht]
	
	\begin{center}
		
		\resizebox{13cm}{!}{
			
			\includegraphics{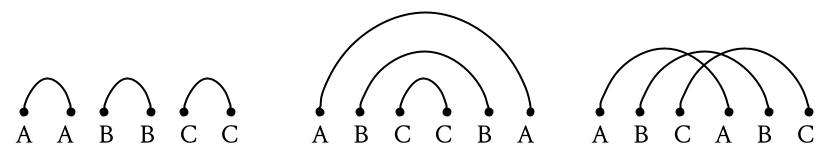}
			
		}
		
		\caption{A line, a stack, and a wave of size three.}
		\label{Matchings3}
	\end{center}
	
\end{figure}

\subsection{In arbitrary matchings}\label{arb_match} Consider a sub-matching $M'$ of $M$ and an edge $e\in M\setminus M'$, which is to the left of the left-most edge $f$ of $M'$.
Note that if $M'$ is a line and $e$ and $f$ form a line, then $M'\cup\{e\}$ is a line too. Similarly, if  $M'$ is a stack and $\{e,f\}$ form a nesting, then $M'\cup\{e\}$ is a stack too.
However, an analogous statement fails to be true for waves, as~$e$, though crossing $f$, may not necessarily cross all other edges of the wave $M'$.
Due to this observation, in the proof of our first result we will need another type of ordered matchings combining lines and waves. We call a matching $M=\{\{a_i,b_i\}:\; i=1,\dots,n\}$ with $a_i<b_i$, for all $i=1,2,\ldots,n$, and $a_1<\cdots<a_n$, \emph{a landscape} if $b_1<b_2<\cdots<b_n$, that is, the right-ends of the edges of $M$ are also linearly ordered (a first-come-first-serve pattern). Notice  that there are no non-trivial stacks in a landscape. In the double occurrence word notation,  a landscape is just a word obtained by a \emph{shuffle} of the two copies of the word $A_1A_2\cdots A_n$. Examples of landscapes for $n=4$ are, among others, $\color{red}{ABCD}\color{blue}{ABCD}$, $\color{red}{A}\color{blue}{A}\color{red}{BC}\color{blue}{BC}\color{red}{D}\color{blue}{D}$, $\color{red}{ABC}\color{blue}{AB}\color{red}{D}\color{blue}{CD}$ (the last one is depicted in Figure~\ref{Matchings10}).

\begin{figure}[ht]
	
	\begin{center}
		
		\resizebox{7cm}{!}{
			\includegraphics{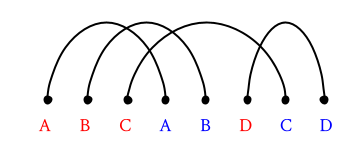}
		}
		
		\caption{A landscape of size four.}
		\label{Matchings10}
	\end{center}
	
\end{figure}

The following is an Erd\H{o}s--Szekeres type result for ordered matchings.

\begin{theorem}\label{Theorem E-S for LSW}
	Let $\ell,s,w$ be arbitrary positive integers and let $n=\ell sw+1$. Then, every ordered matching on $2n$ vertices contains a line of size $\ell+1$, or a stack of size $s+1$, or a wave of size $w+1$.
\end{theorem}

\begin{proof}
Let $M$ be any ordered matching with edges $\{a_i,b_i\}$, $i=1,2,\ldots, n$. Let $s_i$ denote the size of a largest stack whose left-most edge is $\{a_i,b_i\}$. Similarly, let $\lambda_i$ be the largest size of a landscape whose left-most edge is  $\{a_i,b_i\}$. Consider the sequence of pairs $(s_i,\lambda_i)$, $i=1,2,\ldots,n$. We argue that no two pairs of this sequence may be equal. Indeed, let $i<j$ and consider the two edges $\{a_i,b_i\}$ and $\{a_j,b_j\}$. These two edges may form a nesting, an alignment, or a crossing. In the first case we get $s_i>s_j$, since the edge $\{a_i,b_i\}$ enlarges the largest stack starting at $\{a_j,b_j\}$. In the two other cases, we have $\lambda_i>\lambda_j$ by  the same argument. Since the number of pairs $(s_i,\lambda_i)$ is $n>s\cdot \ell w$, it follows that either $s_i>s$ for some $i$, or $\lambda_j>\ell w$ for some $j$. In the first case we are done, as there is a stack of size $s+1$ in $M$.

In the second case, assume that $L$ is a landscape in $M$ of size at least $p=\ell w+1$. Let us order the edges of $L$ as $e_1<e_2<\cdots<e_p$, accordingly to the linear order of their left ends. Decompose $L$ into edge-disjoint waves, $W_1, W_2,\ldots, W_k$, in the following way. For the first wave $W_1$, pick $e_1$ and all edges whose left ends are between the two ends of $e_1$, say, $W_1=\{e_1<e_2<\ldots<e_{i_1}\}$, for some $i_1\geqslant 1$. Clearly, $W_1$ is a true wave since there are no nesting pairs in $L$. Also notice that the edges $e_1$ and $e_{i_1+1}$ are non-crossing since otherwise the latter edge would be included in $W_1$. Now, we may remove the wave $W_1$ from $L$ and repeat this step for $L-W_1$ to get the next wave $W_2=\{e_{i_1+1}<e_{i_1+2}<\ldots<e_{i_2}\}$, for some $i_2\geqslant i_{1}+1$. And so on, until exhausting all edges of $L$, while forming the last wave $W_k=\{e_{i_{k-1}+1}<e_{i_{k-1}+2}<\ldots<e_{i_k}\}$, with $i_k\geqslant i_{k-1}+1$. Clearly, the sequence $e_1<e_{i_1+1}<\ldots<e_{i_{k-1}+1}$ of the leftmost edges of the waves $W_i$ must form a line (see Figure~\ref{Matchings5}).
\begin{figure}[ht]
	
	\begin{center}
		
		\resizebox{15cm}{!}{
			\includegraphics{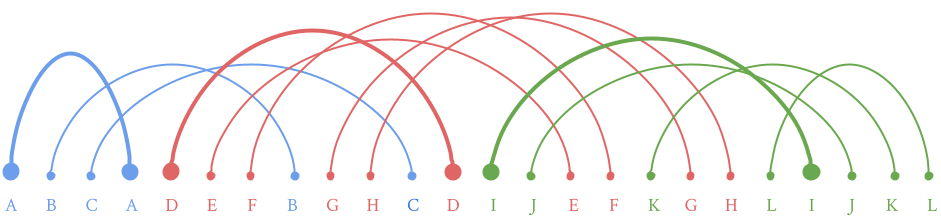}
		}
		
		\caption{Greedy decomposition of a landscape into waves. The left-most edges of the waves (in bold) form a line.}
		\label{Matchings5}
	\end{center}
	
\end{figure}
So, if $k\geqslant \ell+1$, we are done. Otherwise, we have $k\leqslant \ell$, and because $p=\ell w+1$, some wave $W_i$ must have at least $w+1$ edges. This completes the proof.
\end{proof}

It is not hard to see that the above result is optimal. For example, consider the case $\ell=5$, $s=3$, $w=4$. Take $3$ copies of the wave of size $w=4$: $\color{red}{ABCDABCD}$, $\color{blue}{PQRSPQRS}$, $\color{orange}{XYZTXYZT}$. Arrange them into a stack-like structure (see Figure~\ref{Matchings9}): $$\color{red}{ABCD}\color{blue}{PQRS}\color{orange}{XYZTXYZT}\color{blue}{PQRS}\color{red}{ABCD}.$$
\begin{figure}[ht]
	
	\begin{center}
		
		\resizebox{15cm}{!}{
			\includegraphics{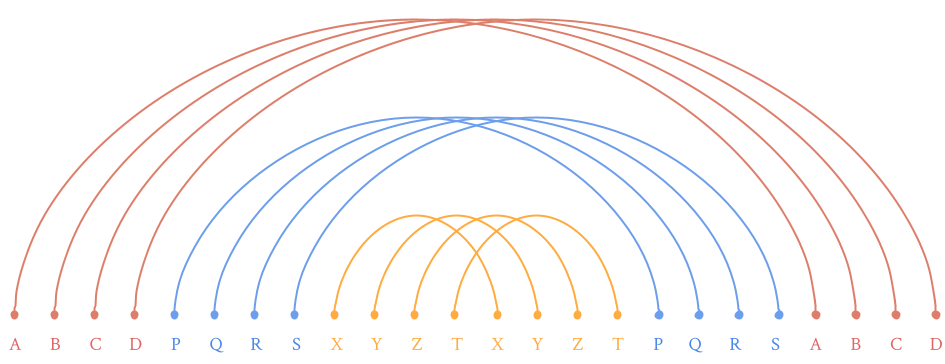}
		}
		
		\caption{A stack of waves.}
		\label{Matchings9}
	\end{center}
	
\end{figure}
Now, concatenate  $\ell=5$ copies of this structure. Clearly, we obtain a matching of size $\ell sw=5\cdot 3\cdot 4$ with no line of size $6$, no stack of size $4$, and no wave of size $5$. This example can be easily generalized to get the following fact.
\begin{proposition}
	For every positive integers $\ell,s$ and $w$ there exists a matching of size $n=\ell sw$ with all lines, waves, and stacks of size at most $\ell,s$ and $w$, respectively.
\end{proposition}

By forbidding one of the three basic structures to be present in $M$ and setting the corresponding parameter, $\ell,s$, or $w$ to 1, we immediately deduce from Theorem \ref{Theorem E-S for LSW} that one of the other two structures of an appropriately large size must be present in $M$. For example, in a landscape (i.e.,~no nestings) of size $n\ge \ell w+1$ one can find either a line of size $\ell+1$ or a wave of size $w+1$. More interestingly, forbidding an alignment we obtain what we define in Section \ref{twins_am} as  permutational matchings which are in a one-to-one correspondence with permutations of order $n$. Moreover, under this bijection waves and stacks in a permutational matching $M$ correspond to, respectively, increasing and decreasing subsequences of the permutation which is the image of $M$. Thus, we recover the original Erd\H os--Szekeres Theorem as a special case of our Theorem \ref{Theorem E-S for LSW}.

Finally, we formulate separately the diagonal case of Theorem \ref{Theorem E-S for LSW}.

\begin{corollary}\label{ESequal}
	Every ordered matching on $2n$ vertices contains a line,  a stack, or a wave of size at least $n^{1/3}$.
\end{corollary}

\begin{proof}
Given $n$. Clearly, $\lceil n^{1/3} \rceil -1< n^{1/3}$. Hence, $(\lceil n^{1/3} \rceil -1)^3 < n$ and so $(\lceil n^{1/3} \rceil-1)^3 + 1\le n$. Applying Theorem~\ref{Theorem E-S for LSW} with $\ell=s=w=\lceil n^{1/3} \rceil -1$ implies one of the three structures of size $(\lceil n^{1/3} \rceil -1)+1 \ge n^{1/3}$.
\end{proof}

\subsection{3-uniform matchings}\label{section:hyper}

 It is perhaps natural and interesting  to try to generalize Theorem \ref{Theorem E-S for LSW} to \emph{$r$-uniform ordered matchings}, that is, families of $n$ disjoint $r$-element subsets of the linearly ordered set $[rn]$. Here we make the first step by showing that the case $r=3$ follows relatively easily from Theorem \ref{Theorem E-S for LSW} itself.

 At the start the problem seems a bit overwhelming, as there are $\tfrac12\binom63=10$ different ways in which two triples may intertwine.
  Using the \emph{triple occurrence words}, they are $AAABBB$, $AABABB$, $AABBBA$, $AABBAB$, $ABBBAA$, $ABBAAB$, $ABBABA$, $ABAABB$, $ABABBA$, $ABABAB$. We will call them \emph{patterns} or \emph{relations} since these words describe a way in which any two edges are interwoven with each other. What is worse, one of them, $AABABB$, stands out as a culprit who spoils the otherwise nice picture. To see it through, call an ordered pair of triples $(e,f)$ \emph{collectable} if for each $k\ge1$ there exists a collection of $k$ triples such that every pair of them is order-isomorphic to the pair $(e,f)$.

  For instance, the relation $AAABBB$, which we may call, as before, \emph{an alignment}, is collectable,  as for any $k$ one can take $A_1A_1A_1A_2A_2A_2\dots A_kA_kA_k$. Similarly, for $AABBBA$, say, one may take
  $(A_1A_1A_2A_2\dots A_kA_k)(A_k\dots A_2A_1)$. In fact, all nine relations but $AABABB$ are collectable. However, for $AABABB$, which we may call \emph{an engagement}, one cannot even add a third triple $CCC$. Indeed, if it were possible, then there should be a $C$ between the second and third $A$, but after the first $B$, which makes the relation between $B$'s and $C$'s not an engagement, as it would begin with $BC$. (Here we assumed w.l.o.g. that the first $A$ precedes the first $B$ which precedes the first $C$.)

  Due to this annoying exception, the Erd\H os--Szekeres type result we are going to prove is not as clean as its predecessors for singletons and pairs. For reasons, which will become clear once we reveal our proof strategy, we ``give names to all the animals'' as  presented in the first two columns of Table~\ref{table:relations}.

  To explain this encoding, let us denote the three basic graph configurations as $\mathtt{L}=AABB$, $\mathtt{S}=ABBA$, and $\mathtt{W}=ABAB$, accordingly to their (alternative) names, that is, line, stack, and wave.
  Now, each of the ten pairwise relations of triples can be uniquely decomposed into a pair of pairs consisting, resp., of the first two vertices in both triples and the last two vertices in both triples. For example, the relation $AABBAB$ decomposes into a line $AABB$ and a wave $ABAB$.
  We express this fact by writing $AABBAB=AABB\oplus ABAB$ and denoting relation $AABBAB$ by $\mathcal{R}_{\mathtt{LW}}$.

   Obviously, for a given 3-uniform relation this decomposition is unique. However, this mapping is not one-to-one as there are only nine different ordered pairs made of three elements. And, indeed,  the alignment $AAABBB$  and the engagement $AABABB$  both decompose into the same the pair $AABB,\;AABB$. To distinguish between them, we add asterisk for the latter, that is, we denote $AAABBB$ by $\mathcal{R}_{\mathtt{LL}}$ while
   $AABABB$ by $\mathcal{R}^*_{\mathtt{LL}}$.

  The last column of Table~\ref{table:relations} displays  this mapping using, again, the words with underlined and overlined letters. The three relations with the first subscript $\mathtt{S}$ differ in that the letters in the second component of the decomposition $\oplus$ are reversed; indeed,   if the pairs of the first two elements of two triples form the sequence $ABBA$, then the pairs of the last two elements can only form sequences $BBAA, BABA$, or $BAAB$. But, let us emphasize that as Gauss words  $BBAA=AABB$, etc.

\begin{table}[h!]
\centering
\begin{tabular}{ |c||c||c| }
\hline

\textsc{relation}& \textsc{in words} & \textsc{decomposition $\oplus$}\\
 \hline
 \hline
${\mathcal R}_{\tt{LL}}$ & $AAABBB$ & $AABB\oplus AABB=\underline{A}\overline{\underline{A}}\overline{A}\underline{B}\overline{\underline{B}}\overline{B}$\\
 \hline
${\mathcal R}^*_{\tt{LL}}$ & $AABABB$ &  $AABB\oplus AABB=\underline{A}\underline{\overline{A}}\underline{B}\overline{A}\underline{\overline{B}}\overline{B}$\\
 \hline
${\mathcal R}_{\tt{LS}}$ & $AABBBA$ &  $AABB\oplus ABBA=\underline{A}\overline{\underline{A}}\underline{B}\overline{\underline{B}}\overline{B}\overline{A}$\\
 \hline
${\mathcal R}_{\tt{LW}}$ & $AABBAB$ &  $AABB\oplus ABAB=\underline{A}\overline{\underline{A}}\underline{B}\overline{\underline{B}}\overline{A}\overline{B}$\\
 \hline
${\mathcal R}_{\tt{SL}}$ & $ABBBAA$ & $ABBA\oplus BBAA=\underline{A}\underline{B}\overline{\underline{B}}\overline{B}\overline{\underline{A}}\overline{A}$\\
 \hline
${\mathcal R}_{\tt{SS}}$ & $ABBAAB$ &  $ABBA\oplus BAAB=\underline{A}\underline{B}\overline{\underline{B}}\overline{\underline{A}}\overline{A}\overline{B}$\\
 \hline
${\mathcal R}_{\tt{SW}}$ & $ABBABA$ &  $ABBA\oplus BABA=\underline{A}\underline{B}\overline{\underline{B}}\overline{\underline{A}}\overline{B}\overline{A}$\\
 \hline
${\mathcal R}_{\tt{WL}}$  & $ABAABB$ & $ABAB\oplus AABB=\underline{A}\underline{B}\overline{\underline{A}}\overline{A}\overline{\underline{B}}\overline{B}$\\
 \hline
${\mathcal R}_{\tt{WS}}$ & $ABABBA$ & $ABAB\oplus ABBA=\underline{A}\underline{B}\underline{\overline{A}}\underline{\overline{B}}\overline{B}\overline{A}$\\
 \hline
${\mathcal R}_{\tt{WW}}$ & $ABABAB$ & $ABAB\oplus ABAB=\underline{A}\underline{B}\underline{\overline{A}}\underline{\overline{B}}\overline{A}\overline{B}$\\
 \hline
\end{tabular}
\caption{Possible relations of two triples and their corresponding decompositions $\oplus$.}
\label{table:relations}
\end{table}

 To cope with  engagements, we will  ``marry'' them with alignments by combining relations ${\mathcal R}_{\mathtt{LL}}$ and ${\mathcal R}^*_{\mathtt{LL}}$ together.
 First, as in the case of pairs, define \emph{a line} as an ordered 3-uniform matching $M$ such that all pairs of triples of $M$ are in relation ${\mathcal R}_{\mathtt{LL}}$, that is, each pair forms an alignment. Call $M$ \emph{a semi-line} if all pairs of triples  are in either relation ${\mathcal R}_{\mathtt{LL}}$ or ${\mathcal R}^*_{\mathtt{LL}}$, that is, they form an alignment or an engagement.

 \begin{proposition}\label{lines_in_semi-lines}
 Every semi-line of size $k$ contains a line of size at least $k/2$.
 \end{proposition}

\proof Let $e_1,\dots,e_k$ be a semi-line. Define an auxiliary graph $G$ on vertex set $[k]$ where $ij\in E(G)$ if $e_i$ and $e_j$ form an engagement. It turns out that $G$ is a linear forest. To prove it, assume w.l.o.g. that for all $i<j$ the left-most vertex of $e_i$ is to the left of the left-most vertex of $e_j$. We claim that  every vertex $i$  has at most one neighbor $j>i$. Indeed, if there were edges $ij$ and $ih$, $i<j<h$, then, identifying $e_i,e_j$ and $e_h$, respectively, with letters $A,B$ and $C$, we would have a sequence $AABCA\dots$, where the last 4 positions are occupied by 2 letters $B$ and 2 letters $C$ in an arbitrary order. This means, however, that $B$'s and $C$'s, or equivalently $e_j$ and $e_h$, form neither an alignment nor an engagement, a contradiction with the definition of a semi-line. By symmetry, any vertex $i$ has also at most one neighbor $j<i$. So, $G$ is, indeed, a linear forest and as such has an independent number at least $k/2$. This completes the proof. \qed

\medskip

We are now ready to formulate an Erd\H os--Szekeres type result for 3-uniform matchings.

\begin{theorem}\label{Theorem E-S for triples}
	Let $a_{XY}$, where $X,Y\in\{\mathtt{L,S,W}\}$, be arbitrary positive integers and let
$$n=\prod_{X,Y}a_{XY}+1.$$ Then, every ordered $3$-uniform matching $M$ on $3n$ vertices either contains a semi-line of size $a_{\mathtt{LL}}+1$, or there exists $(X,Y)\neq(\mathtt{L,L})$ such that $M$ contains a sub-matching of $a_{XY}+1$ triples every two of which are in relation ${\mathcal R}_{XY}$ (as defined in Table~\ref{table:relations}).
\end{theorem}

\begin{remark}
As in Corollary \ref{ESequal} one can show that every 3-uniform ordered matching of size $n$ contains one of the nine  sub-structures  listed in Theorem \ref{Theorem E-S for triples} of size at least $n^{1/9}$.
\end{remark}

\begin{remark}\label{counter}
At the moment we are not able to find a construction showing optimality of Theorem \ref{Theorem E-S for triples}. It was relatively easy in the graph case, as both, stacks and waves, had the \emph{interval} chromatic number equal to 2, which enabled one to superimpose one into another by blowing up the vertices of one of them and filling its edges with  copies of the other (see Figure~\ref{Matchings9} for the superimposition of  a wave of size 4 upon a stack of size 3).  Following, e.g.,~\cite{FJKMV}, we say that an ordered $r$-uniform hypergraph has \emph{interval chromatic number $r$} if it is $r$-partite with the partition sets forming consecutive blocks of the linearly ordered vertex set.

 Now, out of the eight 3-uniform relations, alignments and engagements aside, only four  have interval chromatic number $3$, namely those not having $\mathtt{L}$ within their indices: ${\mathcal R}_{\mathtt{SS}}=ABBAAB$,   ${\mathcal R}_{\mathtt{SW}}=ABBABA$, ${\mathcal R}_{\mathtt{WS}}=ABABBA$, and   ${\mathcal R}_{\mathtt{WW}}=ABABAB$. Thus, we can provide a counterexample showing optimality of Theorem \ref{Theorem E-S for triples} only in the special case when $a_{\mathtt{LS}}=a_{\mathtt{LW}}=a_{\mathtt{SL}}=a_{\mathtt{WL}}=1$.  We do so by superimposing the  four ``no-$\mathtt{L}$'' relations and then taking $a_{\mathtt{LL}}$ disjoint copies of the obtained construction. The superimposition is done by carefully replacing each triple of the current matching with a matching obeying mutually the new relation. E.g.,  replacing $ABABAB$ (relation ${\mathcal R}_{\mathtt{WW}}$)  with $CD\;EF\;CD\;EF\;DC\;FE$ results in  a matching where four pairs of triples retain relation ${\mathcal R}_{\mathtt{WW}}$, while two pairs ($C-D$ and $E-F$) enjoy the new relation~${\mathcal R}_{\mathtt{WS}}$. It is crucial here that consecutive blocks of size 2 in $CD\;CD\;DC$ each contain both letters $C$ and $D$, which is equivalent to having interval chromatic number $3$.

For the case when $a_{\tt{SS}}=a_{\tt{SW}}=a_{\tt{WS}}=a_{\tt{WW}}=2$ (and all other parameters set to 1) see  Figure~\ref{Matchings8}, where the pairwise relations between all 16 triples can be described as follows. Let us first focus on the 8 triples drawn with solid colored lines ($A,\dots,H$). Among them, all 4 monochromatic pairs of triples (like $A-B$)  are in relation ${\mathcal R}_{\mathtt{SS}}$, the 8 red-blue and orange-green pairs of triples (like $A-C$) obey relation ${\mathcal R}_{\mathtt{SW}}$, while the remaining 16 pairs (like $A-E$) satisfy relation ${\mathcal R}_{\mathtt{WS}}$. The very same is true for the 8 triples drawn with dashed colored lines ($I,\dots,P$).  Finally, each solid-dashed pair (like $A-I$) satisfies relation ${\mathcal R}_{\mathtt{WW}}$.
Note that, viewing this construction  as a complete graph $K_{16}$, each relation corresponds to a bipartite subgraph, and so, crucially, no three triples are mutually in the same relation.
\end{remark}
\begin{figure}[ht]
	
	\begin{center}
		
		\resizebox{15cm}{!}{
			\includegraphics{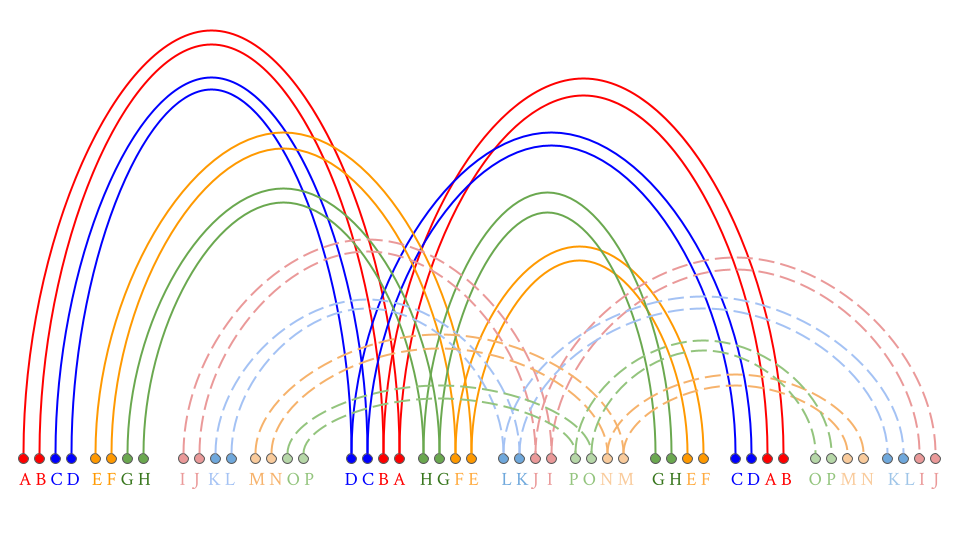}
		}
		
		\caption{A matching of $2^4=16$ triples illustrating optimality of Theorem \ref{Theorem E-S for triples} with $a_{\mathtt{SS}}=a_{\mathtt{SW}}=a_{\mathtt{WS}}=a_{\mathtt{WW}}=2$ (and all other parameters set to~1). }
		\label{Matchings8}
	\end{center}
	
\end{figure}

\begin{remark}
As in the case of graph matchings it is impossible to employ directly the original Erd\H os--Szekeres proof idea. To see this, call a relation ${\mathcal R}$ \emph{extendable} if, given a collection $M$ of disjoint triples  which are mutually in relation ${\mathcal R}$, every triple to the left of $M$ which is in  relation ${\mathcal R}$ with the left-most triple of $M$, is in relation ${\mathcal R}$ with all triples in $M$. For graphs, lines and stacks were extendable, while waves were not, which forced us to come up with the notion of a landscape. Now, only ${\mathcal R}_{\mathtt{LL}}$, ${\mathcal R}_{\mathtt{LS}}$, ${\mathcal R}_{\mathtt{SW}}$, and ${\mathcal R}_{SL}$ are extendable, so any direct proof of Theorem \ref{Theorem E-S for triples} would require more sophisticated analogs of landscapes. Fortunately, we deduce  Theorem \ref{Theorem E-S for triples} directly from Theorem \ref{Theorem E-S for LSW}.
\end{remark}

\begin{proof}[Proof of Theorem \ref{Theorem E-S for triples}]
Let $M$ be a matching (of triples) as in the assumptions of the theorem. Our proof strategy is to apply Theorem \ref{Theorem E-S for LSW} twice, first to the (graph) matching $M_1$ composed of pairs consisting of the first two elements of the triples in $M$, then to a suitably chosen sub-matching of the (graph) matching $M_1$ composed of pairs of the last two elements of the triples in $M$.

Set $b_{\mathtt{L}}=a_{\mathtt{LL}}a_{\mathtt{LS}}a_{\mathtt{LW}}$, $b_{\mathtt{S}}=a_{\mathtt{SL}}a_{\mathtt{SS}}a_{\mathtt{SW}}$, and $b_{\mathtt{W}}=a_{\mathtt{WL}}a_{\mathtt{WS}}a_{\mathtt{WW}}$ and apply Theorem \ref{Theorem E-S for LSW} to $M_1$ with $\ell:=b_{\mathtt{L}}$, $s:=b_{\mathtt{S}}$, and $w:=b_{\mathtt{W}}$. We will now examine the three alternatives of the conclusion.

{\bf Case ($\mathtt{L}$).} $M_1$ contains a (graph) line $L$ of size $b_{\mathtt{L}}+1$. Let $M_L$ be the sub-matching of $M_2$ composed of pairs of the last two elements of those triples in $M$ whose first two elements form a pair belonging to $L$. Formally,
$$M_L=\{\{j,k\}:\; i<j<k,\;\{i,j,k\}\in M,\quad\mbox{and}\quad\{i,j\}\in L\}.$$
We apply Theorem \ref{Theorem E-S for LSW} to $M_L$ with $\ell:=a_{\mathtt{LL}}$, $s:=a_{\mathtt{LS}}$, and $w:=a_{\mathtt{LW}}$ and examine the three possible outcomes.

{\bf Subcase $(\mathtt{LL})$.} $M_L$ contains a line \textit{\L} of size $a_{\mathtt{LL}}+1$. This is the troublesome case. Let $M_{\mathtt{LL}}$ be the sub-matching of $M$ consisting of all triples in $M$ whose first pair of elements belongs to $L$, while the last pair belongs to \textit{\L}. Consider any two triples in $M_{\mathtt{LL}}$. There are two possible relations they may be in: the alignment $\underline{A}\overline{\underline{A}}\overline{A}\underline{B}\overline{\underline{B}}\overline{B}$ and the engagement $\underline{A}\underline{\overline{A}}\underline{B}\overline{A}\underline{\overline{B}}\overline{B}$. Indeed, in both cases the first pairs and the last pairs, underlined and overlined, resp., form (graph) alignments. This shows that $M_{\mathtt{LL}}$ is a semi-line.

{\bf Subcase $(\mathtt{LS})$.} $M_L$ contains a stack ${S}$ of size $a_{\mathtt{LS}}+1$.  Let $M_{\mathtt{LS}}$ be the sub-matching of $M$ consisting of all triples in $M$ whose first pair of elements belongs to $L$, while the last pair belongs to ${S}$. Consider any two triples in $M_{\mathtt{LS}}$. Then they necessarily are in relation ${\mathcal R}_{\mathtt{LS}}$, that is, they form the word $\underline{A}\overline{\underline{A}}\underline{B}\overline{\underline{B}}\overline{B}\overline{A}$.

{\bf Subcase $(\mathtt{LW})$.} $M_L$ contains a wave ${W}$ of size $a_{\mathtt{LW}}+1$. Let $M_{\mathtt{LW}}$ be the sub-matching of $M$ consisting of all triples in $M$ whose first pair of elements belongs to $L$, while the last pair belongs to ${W}$. Consider any two triples in $M_{\mathtt{LW}}$. Then they necessarily are in relation ${\mathcal R}_{\mathtt{LW}}$, that is, they form the word
$\underline{A}\overline{\underline{A}}\underline{B}\overline{\underline{B}}\overline{A}\overline{B}$.

{\bf Cases ($\mathtt{S}$) and ($\mathtt{W}$).} Each case splits further, as before, into three subcases. It is straightforward to check that these altogether six subcases lead to the remaining six relations ${\mathcal R}_{\mathtt{SL}}$, ${\mathcal R}_{\mathtt{SS}}$, ${\mathcal R}_{\mathtt{SW}}$, ${\mathcal R}_{\mathtt{WL}}$, ${\mathcal R}_{\mathtt{WS}}$, ${\mathcal R}_{\mathtt{WW}}$, yielding each time the required size of the collection.
\end{proof}

\subsection{In random matchings}\label{section:random}
In this section we shall investigate the size of unavoidable structures one can find in \emph{random} ordered matchings with the emphasis on the three canonical patterns: lines, stacks, and waves. Recall that  $\rm_{n}$ is a random  (ordered) matching of size $n$, that is, a matching picked uniformly at random out of the set of all $\alpha_n:=(2n)!/(n!2^n)$ matchings on the set  $[2n]$.

Baik and Rains in~\cite{BaRa} (see also~\cite[Theorem 17]{Stanley}) determined the asymptotic distribution of the maximum size of two of the three canonical patterns contained in a random ordered matching. As a consequence, their values can be pinpointed very precisely.
\begin{theorem}[\cite{BaRa}]\label{St}
The sizes of the largest stack and the largest wave contained in $\rm_n$ are a.a.s.\ equal to~$(1+o(1))\sqrt{2n}$.
\end{theorem}
A similar result for lines was proved by Justicz, Scheinerman, and Winkler in \cite{JSW}. Note, however, that the constant is different.

\begin{theorem}[\cite{JSW}]\label{lines}
The size of the largest line contained in $\rm_n$ is a.a.s.\ equal to $(2+o(1))\sqrt{n/\pi}$.
\end{theorem}

In this section we provide  simpler, purely combinatorial proofs of  weaker versions of  Theorems \ref{St} and \ref{lines}, with the asymptotic coefficient, resp., $\sqrt2$ and $2/\sqrt\pi$ replaced by  pairs of constants setting  lower and upper bounds only.
 The proof of the  upper bounds is quite straightforward and provides a more general result.

\begin{proposition}\label{upb}
Let $(M_k)_1^\infty$ be a sequence of ordered matchings of size $k$, $k=1,2,\dots$. Then, a.a.s.
$$\max\{k:\; M_k\preceq \rm_n\}\le (1+o(1))e\sqrt{  n/2}.$$
\end{proposition}
\begin{proof}
Set $k_0=\lfloor \left(1+n^{-1/3}\right)e\sqrt{ n/2}\rfloor$, and let $X_k$ be a random variable counting the number of ordered copies of $M_k$ in $\rm_n$. Our goal is to show, via the first moment method, that a.a.s.\ $X_k=0$ for all $k\ge k_0$. Note that it is \emph{not} enough to show that a.a.s.\ $X_{k_0}=0$, as the sequence $(M_k)_1^\infty$ may not be ascending.

To compute the expectation of $X_k$, one has to first choose the $2k$ vertices of a copy of $M_k$ (the copy itself is  placed in just one way), then count the number of extensions of that copy to an entire matching of size $n$, and, finally, divide  by the total number of matchings. Thus,
\[
\E X_k = \binom{2n}{2k} \cdot 1 \cdot \frac{\alpha_{n-k}}{\alpha_n}
= \frac{2^k}{(2k)!} \cdot \frac{n!}{(n-k)!}
\le \frac{2^k}{(2k)!} \cdot n^k
\le \frac{2^k}{(2k/e)^{2k}} \cdot n^k
= \left( \frac{e^2n}{2k^2} \right)^k,
\]
and, by the union bound applied together with Markov's inequality,
\begin{align*}
\Pr(\exists k\ge k_0:\; X_k>0) &\le\sum_{k=k_0}^{n}\Pr(X_k\ge1)\\
&\le\sum_{k=k_0}^{n}\E X_k\le \sum_{k=k_0}^{n}\left( \frac{e^2n}{2k^2} \right)^k\le n(1+n^{-1/3})^{-2k_0}=o(1).
\end{align*}
\end{proof}

Owing to the very homogeneous structure of lines, stacks, and waves, we are able to establish  corresponding  lower bounds for their maximum sizes in $\rm_n$. It is, perhaps, interesting to note that, unlike for permutations, the size of the sub-structures guaranteed by the of Erd\H os--Szekeres-type result (cf. Corollary \ref{ESequal}) grows substantially in the random setting.

\begin{theorem}\label{lowerLSW}
A random matching $\rm_{n}$ contains a.a.s.\ stacks and waves of size at least
$\tfrac{1-o(1)}{e\sqrt 2}\sqrt{n}$ each, as well as lines of size at least $\tfrac18\sqrt n$.
\end{theorem}

The proof for stacks and waves is very simple and relies mostly on Theorem \ref{Theorem E-S for LSW}.
For lines we will  make use of the following lemma which might be of some independent interest. For that reason we state it in a more general setting than what we actually need.
The \emph{length} of an edge $\{i,j\}$ in a matching on $[2n]$ is defined as $|j-i|$.

\begin{lemma}\label{lemma:edges}
Let a sequence $f(n)$ be such that $f(n)\to\infty$ and $f(n)=o(n)$. 
Then, a.a.s.~the number of edges of length at most $f(n)$ in $\rm_{n}$ is $(1+o(1))f(n)$.
\end{lemma}
\begin{proof}
For each pair of vertices $1\le u<v\le 2n$, let $X_{uv}$ be the indicator random variable  equal to one if $\{u,v\}\in\rm_{n}$ and 0 otherwise. Clearly, $\Pr(X_{uv}=1)=1/(2n-1)$. Let us write for simplicity $m=f(n)$. The sum $X = \sum_{1\le v-u \le m} X_{uv}$ counts all edges in $\rm_n$ of length at most $m$.
As the number of summands is $(2n-1)+(2n-2)+\dots+(2n-m)= 2nm-\binom{m+1}2$,  we have
\[
\E X = \frac{2nm-\binom{m+1}2}{2n-1} =m(1-O(m/n)).
\]
In particular, by the assumptions on $f(n)$, we get $\E X=m(1-o(1))\to\infty$. To estimate the second moment, observe that
\[
\Pr(X_{u_1v_1}=X_{u_2v_2}=1) =
\begin{cases}
\frac{1}{(2n-1)(2n-3)}, & \text{ if } \{u_1,v_1\}\cap \{u_2,v_2\} = \emptyset,\\
0, & \text { otherwise}.
\end{cases}
\]
Hence, estimating quite crudely,
\[
\E(X(X-1)) \le \frac{\left(2nm-\binom{m+1}2\right)^2}{(2n-1)(2n-3)},
\]
and consequently, by Chebyshev's inequality
\begin{align*}&\Pr(|X-\E X|\ge\eps\E X)\le \frac{\E (X(X-1))+\E X-(\E X)^2}{\eps^2(\E X)^2}\\&\qquad\qquad\le\frac1{\eps^2}\left(\frac{(2n-1)^2}{(2n-1)(2n-3)}+\frac1{\E X}-1\right)=\frac1{\eps^2}\left(\frac2{2n-3}+\frac1{\E X}\right)\to0,
\end{align*}
provided 
that $\eps^2 m \to\infty$.
This implies that a.a.s.\ $1-\eps-m/(2n)<X/m\le 1+\eps$.
\end{proof}

\medskip

\begin{proof}[Proof of Theorem \ref{lowerLSW}] We split the proof into two uneven parts.

\textbf{Stacks and waves:} Fix $\eps>0$ and choose $\eps'>0$ such that $\tfrac{1-2\eps'}{1+\eps'}=1-\eps$. Let $M$ be the sub-matching of $\rm_n$ consisting of all edges with one end in $[n]$ and the other in $[2n]\setminus[n]$. Set $X:=|M|$. We have $\E X=n^2/(2n-1)\sim n/2$ and
\[
\E (X(X-1))=\frac{n^2(n-1)^2}{(2n-1)(2n-3)}\sim (\E X)^2,
\]
so, using Chebyshev's inequality, it follows that  a.a.s.\ $|X-n/2|=o(n)$.

 Observe that there are no lines in $M$ longer than 1. By Proposition~\ref{upb}, a.a.s.\ there are no stacks or waves in $\rm_n$ of size  $k_0=\lfloor (1+o(1)) e\sqrt{ n/2}\rfloor$. Hence, by applying Theorem~\ref{Theorem E-S for LSW} to $M$  with $\ell=1$, $s=k_0-1$, and $w=\lfloor\tfrac{X-1}{k_0-1}\rfloor$, we conclude that a.a.s.\ there is a wave in $M$, and thus in $\rm_n$, of size at least
 $$w+1=\frac{1-o(1)}{e\sqrt 2}\sqrt{n}.$$
By swapping the roles of waves and stacks in the above argument, we deduce that a.a.s.\ there is also a stack in~$M$, and thus in $\rm_n$, of size  $\frac{1-o(1)}{e\sqrt 2}\sqrt{n}$.

\textbf{Lines:}
Let $m = \lfloor\sqrt{n}/2\rfloor$. By Lemma~\ref{lemma:edges}, a.a.s.~the number of edges of length at most~$m$ in $\rm_{n}$ is at least $\sqrt{n}/4$. We will show that among the edges of length at most~$k$, there are a.a.s.\ at most $\sqrt{n}/8$ pairs forming crossings or nestings. After removing one edge from each crossing and nesting we obtain a line of size at least $\sqrt{n}/4 - \sqrt{n}/8 = \sqrt{n}/8$.

For a 4-element subset $S=\{u_1,u_2,v_1,v_2\}\subset [2n]$ with $u_1<v_1<u_2<v_2$, let $X_S$ be an indicator random variable equal to $1$ if both $\{u_1,u_2\}\in \rm_n$ and $\{v_1,v_2\}\in \rm_n$, that is, if $S$ spans a crossing in $\rm_n$. Clearly,
\[
\Pr(X_S=1) = \frac{1}{(2n-1)(2n-3)}.
\]

Let $X = \sum X_S$, where the summation is taken over all sets $S$ as above and such that $u_2-u_1\le m$ and $v_2-v_1\le m$.
Note that this implies that $v_1-u_1\le m-1$. Let $f(n,m)$ denote the number of terms in this sum.
 We have
$$f(n,m)\le\left(2n(m-1)-\binom{m}2\right)\binom m2\le\left(nm-\frac12\binom{m}2\right)m^2,$$
as we have at most $2n(m-1)-\binom{m}2$ choices for $u_1$ and $v_1$ (see the proof of Lemma \ref{lemma:edges} with $m$ replaced by $m-1$) and, once $u_1, v_1$ have been selected, at most $\binom m2$ choices  of $u_2$, and $v_2$.
It is easy to see that $f(n,m)=\Omega(nm^3)$. (In fact, one could show that $f(n,m)\sim \tfrac23nm^3$, but we do not care about optimal constants here.) Hence,  $\E X = \Omega(m^3/n)\to\infty$, while
$$\E X =\sum_S \E X_S=\frac{f(n,m)}{(2n-1)(2n-3)} \le \frac{m^3}{4n}= \frac1{32}\sqrt{n}.$$

To apply Chebyshev's inequality, we need to estimate $\E (X(X-1))$, which can be written as
\[
\E(X(X-1))=\sum_{S,S'}\Pr\left(\{\{u_1,v_1\},\{u_2,v_2\},\{u_1',v_1'\},\{u_2',v_2'\}\}\subset\rm_n\right),
\]
where the summation is taken over all (ordered) pairs of sets $S=\{u_1,u_2,v_1,v_2\}\subset [2n]$ with $u_1<v_1<u_2<v_2$ and $S'=\{u'_1,u'_2,v'_1,v'_2\}\subset [2n]$ with $u'_1<v'_1<u'_2<v'_2$  such that $u_2-u_1\le m$, $v_2-v_1\le m$, $u'_2-u'_1\le m$, and $v'_2-v'_1\le m$.
 We split the above sum into two sub-sums $\Sigma_1$ and $\Sigma_2$ according to whether $S\cap S'=\emptyset$ or $|S\cap S'|=2$ (for all other options the above probability is zero). In the former case,
 $$\Sigma_1\le\frac{f(n,m)^2}{(2n-1)(2n-3)(2n-5)(2n-7)}=(\E X)^2(1+O(1/n)).$$
 In the latter case, the number of such pairs $(S,S')$ is at most $f(n,m)\cdot 4m^2$, as given $S$, there are four ways to select the common pair and at most $m^2$ ways to select the remaining two vertices of $S'$. Thus,
 $$\Sigma_2\le \frac{f(n,m)\cdot 4m^2}{(2n-1)(2n-3)(2n-5)}=O(m^5/n^2)=O(\sqrt n)$$
 and, altogether,

\[
\E (X(X-1)) \le (\E X)^2(1+O(1/n)) + O(\sqrt n)=(\E X)^2 + O(\sqrt n).
\]
By Chebyshev's inequality,
\begin{align*}\Pr(|X-\E X|\ge\E X)&\le \frac{\E(X(X-1))+\E X-(\E X)^2}{(\E X)^2}\\&\le1+O(1/\sqrt n)+\frac1{\E X}-1=O(1/\sqrt n)\to0.
\end{align*}
 Thus, a.a.s.\ $X\le 2\E X\le \sqrt n/16$.

We deal with nestings in a similar way. For a 4-element subset $S=\{u_1,u_2,v_1,v_2\}\subset [2n]$ with $u_1<v_1<v_2<u_2$, let $Y_S$ be an indicator random variable equal to $1$ if both $\{u_1,u_2\}\in \rm_n$ and $\{v_1,v_2\}\in \rm_n$, that is, if $S$ spans a nesting in $\rm_n$.
 Further, let $Y = \sum Y_{S}$, where the summation is taken over all sets $S$ as above and such that $u_2-u_1\le m$ and (consequently) $v_2-v_1\le m-2$.
It is crucial to observe that, again, $\E Y \le m^3/n = \sqrt{n}/32$.
Indeed, this time there are  at most $2nm-\binom{m+1}2$ choices for $u_1$ and $u_2$ and, once $u_1, u_2$ have been selected, at most $\binom {m-2}2$ choices  of $v_1$, and $v_2$, while the probability of both pairs appearing in $\rm_n$ remains the same as before.
The remainder of the proof goes mutatis mutandis.

We conclude that a.a.s.~the number of crossings and nestings of length at most $m$ in $\rm_n$ is at most  $ \sqrt{n}/8$ as was required.
\end{proof}

We close this section with a straightforward generalization of Proposition~\ref{upb} to
 random $r$-uniform ordered matchings, $r\ge2$. Let $\rm^{(r)}_{n}$ be a random  (ordered) $r$-matching of size $n$, that is, a matching picked uniformly at random out of the set of all $\alpha^{(r)}_n:=(rn)!/(n!(r!)^n)$ matchings on the set  $[rn]$.
 \begin{proposition}\label{upbr}
Let $(M^{(r)}_k)_1^\infty$ be a sequence of ordered $r$-matchings of size $k$, $k=1,2,\dots$. Then, a.a.s.
\[
\max\{k:\; M^{(r)}_k\preceq \rm^{(r)}_n\}\le (1+o(1))\frac{e}{r}(r!n)^{1/r}.
\]
\end{proposition}

\proof Let $k_0=\lfloor (1+n^{-1/(r+1)})\frac{e}{r}(r!n)^{1/r}\rfloor$ and, for each $k\ge k_0$, let $X^{(r)}_k$ be the number of ordered copies of $M^{(r)}_k$ of size $k$ in $\rm^{(r)}_n$. Then
\[
\E X^{(r)}_k = \binom{rn}{rk} \cdot 1 \cdot \frac{\alpha^{(r)}_{n-k}}{\alpha^{(r)}_n}
= \frac{(r!)^k}{(rk)!} \cdot \frac{n!}{(n-k)!}
\le \frac{(r!)^k}{(rk)!} \cdot n^k
\le \frac{(r!)^k}{(rk/e)^{rk}} \cdot n^k
= \left( \frac{e^r r! n}{(rk)^r} \right)^k,
\]
which implies that
\[
\Pr(\exists k\ge k_0:\; X^{(r)}_k>0)
\le \sum_{k=k_0}^{n} \left( \frac{e^r r! n}{(rk)^r} \right)^k
\le n(1+n^{-1/(r+1)})^{-rk_0} = o(1).
\]
\qed

In particular, the above statement yields that in $\rm^{(3)}_n$ a.a.s.\ none of the nine collectable sub-matchings (as defined in Subsection~\ref{section:hyper}) has  size bigger than~$2n^{1/3}$. It is, at the moment, an open problem to match it  with a fair lower bound (see discussion in Section \ref{fr}).

\section{Twins}

Recall that for $r\ge2$, by \emph{$r$-twins} in a matching $M$ we mean any collection of disjoint, pairwise  order-isomorphic sub-matchings  $M_1,M_2,\dots,M_r$. For instance, the matching $M=AABCDDEBCFGHIHEGFI$ contains $3$-twins formed by the triple $M_1=\color{blue}{BDDB}$, $M_2=\color{Red}{EHHE}$, and $M_3=\color{ForestGreen}{FGGF}$ (see Figure \ref{Matchings7}).

 Recall also that by the \emph{size} of $r$-twins we mean the size (the number of edges) in just one of them. Let $t_r(M)$  denote the maximum size of $r$-twins in a matching $M$ and $t_r^{\match}(n)$ -- the minimum of $t_r(M)$ over all matchings~$M$ on $[2n]$.

\begin{figure}[ht]
	
	\begin{center}
		
		\resizebox{15cm}{!}{
			
			\includegraphics{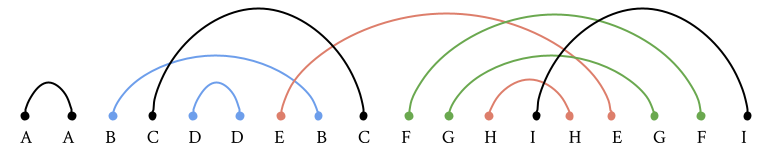}
			
		}
		
		\caption{3-twins of size two.}
		\label{Matchings7}
	\end{center}
	
\end{figure}

 \subsection{In arbitrary matchings}\label{twins_am}
 Let us first point to a direct connection between twins in permutations and ordered twins in a certain kind of matchings.
By a \emph{permutation} we mean any finite sequence of pairwise distinct positive integers. We say that two permutations $(x_1,\dots,x_k)$ and $(y_1,\dots,y_k)$ are \emph{similar} if their entries preserve the same relative order, that is, $x_i<x_j$ if and only if $y_i<y_j$ for all $1\leqslant i<j\leqslant k$. 
 For $r\ge2$, any $r$ pairwise similar and disjoint sub-permutations of a permutation $\pi$ are called \emph{$r$-twins}. For example, in permutation $$(6,\colorbox{cyan}{1},\colorbox{cyan}{4},7,\colorbox{Lavender}{3},9,\colorbox{Lavender}{8},\colorbox{cyan}{2},\colorbox{Lavender}{5}),$$ the red and blue subsequences form a pair of twins, or 2-twins, of length $3$, both similar to the permutation $(1,3,2)$.

 Let $t_r(\pi)$ denote the maximum length of  $r$-twins in a permutation $\pi$ and
  $t_r^{\perm}(n)$ -- the minimum of $t_r(\pi)$ over all permutations~$\pi$ of~$[n]$.
  Gawron \cite{Gawron} proved that $t_2^{\perm}(n)\leqslant cn^{2/3}$ for some constant $c>0$. This was easily generalized to $t_r^{\perm}(n)\leqslant c_rn^{r/(2r-1)}$ (see \cite{DGR}).

  As for a lower bound, notice that by the Erd\H{o}s--Szekeres Theorem, we have $t_r^{\perm}(n)\geqslant \left\lfloor \frac{1}{r}n^{1/2}\right\rfloor$.
   For $r=2$, this bound was substantially improved by Bukh and Rudenko \cite{BukhR}.

 \begin{theorem}[Bukh and Rudenko \cite{BukhR}]\label{Lemma BR} For all $n$,
 	 $t_2^{\perm}(n)\geqslant\frac{1}{8}n^{3/5}$.
 \end{theorem}
 Using their approach, in \cite{DGR1}, we generalized this bound to arbitrary $r\ge3$.
 \begin{theorem}[\cite{DGR1}]\label{Lemma BR-gen} For all $n$ and $r\ge3$,
 	 $t_r^{\perm}(n)\geqslant\frac{1}{3r}n^{\frac R{2R-1}}$, where $R=\binom{2r-1}r$.
 \end{theorem}

 We call an ordered matching $M$ on the set $[2n]$ \emph{permutational} if the left end of each edge of $M$ lies in the set $[n]$ (and so, the right end of each edge lies in $[n+1,2n]$. In the double occurrence word notation such a matching can be written as $M=A_1A_2\ldots A_nA_{i_1}A_{i_2}\ldots A_{i_n}$, where $\pi_M = (i_1,i_2,\ldots, i_n)$ is a permutation of $[n]$. There are only $n!$ permutational matchings, nevertheless this connection to permutations turned out to be quite fruitful. Indeed, it is not hard to see that ordered $r$-twins in a permutational matching $M$ are in one-to-one correspondence with $r$-twins in $\pi_M$. In particular, we have $t_r(M)=t_r(\pi_M)$ for such matchings and, consequently, $t_r^{\match}(n)\le t_r^{\perm}(n)$. In particular, by the above mentioned result of Gawron,  $t_2^{\match}(n)=O(n^{2/3})$.

 More subtle is the opposite relation.
 \begin{proposition}\label{general}
 For all $r\ge2$ and $1\le m\le n$, where $n-m$ is even,
 $$t_r^{\match}(n)\ge \min\left\{t_r^{\perm}(m),\; 2t_r^{\match}\left(\frac{n-m+2}2\right)\right\}.$$
 \end{proposition}
 \begin{proof} Let $M$ be a matching on $[2n]$.
  Split the set of vertices of $M$ into two halves, $A=[n]$ and $B=[n+1,2n]$ and let $M':=M(A,B)$ denote the set of edges of $M$ with one end in $A$ and the other end in $B$. Note that $M'$ is a permutational matching. We distinguish two cases. If $|M'|\geqslant m$, then
  $$t_r(M)\ge t_r(M')=t_r(\pi_{M'})\ge t_r^{\perm}(|M'|)\ge t_r^{\perm}(m).$$
  If, on the other hand, $|M'| < m$, that means, $|M'|\le m-2$, due to the assumption of $n-m$ being even,
  then we have sub-matchings $M_A$ and $M_B$ of $M$ of size at least $(n-m+2)/2$ in  sets, respectively, $A$ and $B$. Thus, in this case, by concatenation,
  $$t_r(M)\ge t_r(M_A)+t_r(M_B)\ge 2t_r^{\match}\left(\frac{n-m+2}2\right).$$
 \end{proof}

Proposition \ref{general} allows, under some mild conditions, to ,,carry over'' any lower bound on $t_r^{\perm}(n)$ to one on $t_r^{\match}(n)$. In view of Theorems \ref{Lemma BR} and \ref{Lemma BR-gen}, as well as the upper bound on  $t_r^{\perm}(n)$ from \cite{DGR} mentioned above, we may assume that the parameter $\alpha$ introduced in the lemma below falls between $R/(2R-1)$ and $r/(2r-1)$.

 \begin{lemma}\label{polyn} For all $r\ge2$, $R/(2R-1)\le\alpha\le  r/(2r-1)$, and $\beta>0$, if $t_r^{\perm}(n)\ge \beta n^{\alpha}$ for all $n\ge r$, then $t_r^{\match}(n)\ge \beta( n/4)^\alpha$ for  all $n\ge r$.
 \end{lemma}
\begin{proof} Assume that for some $r\ge2$ and $\alpha,\beta$ as above,  $t_r^{\perm}(n)\ge \beta n^{\alpha}$, for all $n\ge r$. We will prove that $t_r^{\match}(n)\ge \beta(n/4)^\alpha$  by induction on $n$. For $n\leqslant 4\left(\tfrac 1\beta\right)^{1/\alpha}$ the claimed bound is at most $1$, so it is trivially true. Assume then that $n\ge \left(\tfrac 1\beta\right)^{1/\alpha}$ and that $t_r^{\match}(n')\ge \beta(\gamma n')^\alpha$ for all $r\le n'<n$.
Let $n'\in\{\lceil n/4\rceil, \lceil n/4\rceil+1\}$ have the same parity as $n$.  Then, by Proposition \ref{general} with $m=n'$,
$$t_r^{\match}(n)\ge \min\left\{t_r^{\perm}(n'), 2t_r^{\match}\left(\frac{n-n'+2}2\right)\right\}.$$
Now, by the assumption of the lemma applied to $n=r$, noticing that $t_r^{\match}(r)=1$, we have that $1\ge\beta r^\alpha$, or $\beta\le r^{-\alpha}$, so $n'\ge n/4\ge r$.
Thus, by the assumption of the lemma applied to $n'$, $t_r^{\perm}(n')\ge \beta(n/4)^\alpha$.
Further, noticing that $n-n'+2\ge 3n/4$,
again by the induction assumption, we also have
\begin{align*}2t_r^{\match}\left(\frac{n-n'+2}2\right)\ge2\beta\left(\frac14\cdot\frac38n\right)^\alpha\ge\beta(n/4)^\alpha,
\end{align*}
where the last inequality is equivalent to $3\ge2^{3-1/\alpha}$ which, in turn, follows by estimating the R-H-S by $2^{3-(2r-1)/r}\le2^{3-3/2}=2^{3/2}$.
\end{proof}

In particular,  Theorem \ref{Lemma BR} and Lemma \ref{polyn} with $\beta=1/8$, $\alpha=3/5$, and $\gamma=1/4$ imply immediately the following result.

\begin{corollary}
	For every $n$,  $t_2^{\match}(n)\ge\frac{1}{8}\left(\frac{n}{4}\right)^{3/5}$. \qed
\end{corollary}

Moreover, any future improvement of the bound in Theorem \ref{Lemma BR} would automatically yield a corresponding improvement of the lower bound on
 $t_2^{\match}(n)$.

\subsection{In random matchings}\label{twins_rm}

In this section we study the size of the largest $r$-twins in a random (ordered) matching $\rm_{n}$, which is, recall,  selected uniformly at random from all $\alpha_n:=(2n)!/(n!2^n)$ matchings on vertex set $[2n]$.
The first moment method yields that

\begin{equation}\label{up}
\mbox{a.a.s.}\quad t_r(\rm_n)<cn^{r/(2r-1)}\quad\mbox{for any}\quad c>e2^{-(r-1)/(2r-1)}.
 \end{equation}
 Indeed, the expected number of $r$-twins of size $k$ in $\rm_{n}$ is
\[
\frac1{r!}\binom{2n}{\underbrace{2k,\dots,2k}_{\text{\normalfont{$r$ times}}},2n-2kr}\frac{\alpha_k\cdot 1^{r-1}\cdot \alpha_{n-rk}}{\alpha_{n}}=\frac{2^{(r-1)k}n!}{r!(2k)!^{r-1}k!(n-rk)!}<\left(\frac{e^{2r-1}n^r}{2^{r-1}k^{2r-1}}\right)^k,
\]
where we used  inequalities $n! / (n-rk)! \le n^{rk}$, $(2k)! \ge (2k/e)^{2k}$, $k! \ge (k/e)^k$, and $r! \ge 1$.
Thus, it converges to 0, as $n\to\infty$,
provided $k\ge cn^{r/(2r-1)}$, for any $c$ as in~\eqref{up}.

It turns out that the a.a.s.~the lower bound on $t_r(\rm_n)$ is of the same order.

\begin{theorem}\label{rg1}
For every $r\ge2$, a.a.s.,
\[
t_r(\rm_n) = \Theta\left(n^{r/(2r-1)}\right).
\]
\end{theorem}

In the proof of the lower bound we are going to use the Azuma-Hoeffding inequality for random permutations (see, e.g., Lemma 11 in~\cite{FP} or  Section 3.2 in~\cite{McDiarmid98}). Let us recall that $\Pi_n$ stands for a \emph{random} permutation selected uniformly  from all $(2n)!$ permutations of the set $[2n]$.
\begin{theorem}\label{azuma}
 Let $h(\pi)$ be a function defined on the set of all permutations of order $2n$ such that, for some constant $c>0$, if a permutation $\pi_2$ is obtained from a permutation $\pi_1$ by swapping two elements, then $|h(\pi_1)-h(\pi_2)|\le c$.
Then, for every $\eta>0$,
\[
\PP(|h(\Pi_n)-\E[h(\Pi_n)]|\ge \eta)\le 2\exp(-\eta^2/(4c^2n)).
\]
\end{theorem}

\begin{proof}[Proof of Theorem \ref{rg1}] In view of \eqref{up}, it suffices to prove a lower bound, that is, to show that a.a.s.~$\rm_n$ contains $r$-twins of size $\Omega(n^{r/(2r-1)})$. In doing so we are following the proof scheme applied in \cite[Theorem 1.2]{DGR1}. Set
$$a:=r!^{1/(2r-1)}(2n)^{(r-1)/(2r-1)}$$
 and assume for simplicity that both, $a$ and $N:=2n/a=r!^{-1/(2r-1)}(2n)^{r/(2r-1)}$, are integers. Partition  $[2n]=A_1\cup\cdots\cup A_{N}$, where $A_i$'s are consecutive blocks of  $a$ integers each,  and define, for every $1\le i<j\le N$, a random variable $X_{ij}$ which counts the number of edges of $\rm_n$ with one endpoint in $A_i$ and the other in $A_j$. Consider an auxiliary graph $G:=G(\rm_n)$ on vertex set $[N]$ where $\{i,j\}\in G$ if and only if $X_{ij}\ge r$.

A crucial observation is that a matching of size $t$ in $G$ corresponds to $r$-twins in $\rm_n$ of size $t$. Indeed,  let $M=\{i_1j_1,\dots, i_tj_t\}$, $i_1<\cdots<i_t$, be a matching in $G$. For every $1\le s\le t$, let  $u^s_h\in A_{i_s}$ and $v^s_h\in A_{j_s}$, $h=1,\dots,r$, be such that $e^s_h:=\{u^s_h,v^s_h\}\in\rm_n$. Then, the sub-matchings $\{e^1_1,\dots,e^t_1\},\dots,\{e^1_r,\dots,e^t_r\}$, owing to the sequential choice of $A_i$'s form $r$-twins in $\rm_n$. Thus, our ultimate goal is to show that a.a.s.~$G$ contains a matching of size $\Theta(n^{r/(2r-1)})$.

Let $\nu(G)$ be the largest size (as the number of edges) of a matching in $G$. We will first show that $\nu(G)$ is sharply concentrated around its expectation. For this we appeal to the permutation scheme of generating $\rm_n$ and apply Theorem \ref{azuma}.
For a permutation $\pi$ of $[2n]$, let $M(\pi)=\{\pi(1)\pi(2),\dots,\pi(2n-1)\pi(2n)\}$ be the corresponding matching.
 Further, let $h(\pi)=\nu(G(M(\pi)))$.
Observe that if $\pi_2$ is obtained from a permutation $\pi_1$ by swapping some two of its elements, then at most two edges of $M(\pi_1)$ can be destroyed and at most two edges of $M(\pi_1)$ can be created, and thus the same can be said about the edges of $G(M(\pi_1))$. This, in turn, implies that the size of the largest matching has been altered by at most two, that is, $|h(\pi_1)-h(\pi_2)|\le2$. Hence, Theorem~\ref{azuma} applied to $h(\pi)$, with $c=2$, and, say, $\eta=n^{2r/(4r-1)}$ implies that
\begin{equation}\label{concentrat}
\PP\left(|\nu(G)-\E[\nu(G)]|\ge n^{2r/(4r-1)}\right)=\PP\left(|h(\Pi_n)-\E[h(\Pi_n)]|\ge n^{2r/(4r-1)}\right)=o(1).
\end{equation}
Note that, crucially, $n^{2r/(4r-1)}=o(n^{r/(2r-1)})$, and it thus remains to show that $\E(\nu(G))=\Omega(n^{r/(2r-1)})$.

Let us first estimate from below the probability of an edge in $G$, that is, $\PP(X_{ij}\ge r)$. Trivially,
$\PP(X_{ij}\ge r) \ge\PP(X_{ij}=r)$. We are going to further
bound $\PP(X_{ij}=r)$ from below by counting matchings on $[2n]$ with precisely $r$ edges between the sets $A_i$ and $A_j$, but with no edges within $A_i$ or $A_j$ (the latter is a simplifying restriction).
To build such a  matching one has to first select subsets $S_i,S_j,T_i,T_j$ such that $S_i\subset A_i$, $S_j\subset A_j$, $|S_i|=|S_j|=r$, while
$T_i\cap (A_i\cup A_j)=\emptyset$, $T_j\cap (A_i\cup A_j)=\emptyset$, $T_i\cap T_j=\emptyset$, and $|T_i|=|T_j|=a-r$.
The total number of these selections is
$$ \binom ar^2\binom{2n-2a}{a-r}\binom{2n-3a+r}{a-r}.$$
Then one has to match $S_i$ with $S_j$, $A_i\setminus S_i$ with $T_i$, and $A_j\setminus S_j$ with $T_j$, and find a perfect matching  on the set $[2n]\setminus(A_i\cup A_j\cup T_i\cup T_j)$ of the remaining $n-4a+2r$ elements. There are
$$r!(a-r)!(a-r)!\alpha_{n-2a+r}$$
ways of doing so.

Multiplying the two products together and dividing by $\alpha_n$, we obtain the inequality
$$\PP(X_{ij}=r)\ge\frac{a!^24^a(2n-2a)!n!}{r!2^r(a-r)!^2(2n)!(n-2a+r)!}.$$
Since $a=o(\sqrt n)$, we have
$$(2n)!/(2n-2a)! = (1+o(1)) (2n)^{2a}\quad\mbox{and}\quad n!/(n-2a+r)! = (1+o(1))n^{2a-r},$$
and so the R-H-S above equals
$$\frac{(1+o(1))}{r!}\left(\frac{a^2}{2n}\right)^r.$$
Consequently, for $n$ large enough,
$$\PP(X_{ij}\ge r)\ge\PP(X_{ij}=r)\ge\frac1{2r!}\left(\frac{a^2}{2n}\right)^r.$$

Having estimated the probability of an edge in $G$, we are now in position to  estimate the degree of a vertex. For each $i\in[N]$, let $Y_i=\deg_G(i)$ be the degree of vertex $i$ in $G$. Then
\begin{equation}\label{EY}
\E(Y_i)=\left(N-1\right)\PP(X_{ij}\ge r)\ge \frac N{4r!}\left(\frac{a^2}{2n}\right)^r=\frac{2n}{4r!a}\left(\frac{a^2}{2n}\right)^r=\frac 1{4}.
\end{equation}

There is an obvious bound on the size of the largest matching $\nu(G)$ in $G=(V,E)$ in terms of the vertex degrees, namely
$$\nu(G)\ge\frac{|E(G)|}{2\Delta_G}=\frac{\sum_{i=1}^NY_i}{4\Delta_G},$$
where $\Delta_G$ is the maximum degree in $G$. Note that, trivially,  $\Delta_G\le \min\{a/r,N-1\}=a/r$. Unfortunately, since the expected degrees $\E(Y_i)$ are (bounded by) constants -- cf. \eqref{EY}, in this form the bound on $\nu(G)$ is of no use, as one cannot show concentration of all degrees~$Y_i$ simultaneously. Instead, we resort to an even weaker, but more manageable bound.

For an integer $D>0$, let $G_D$ be a subgraph of $G$ induced by the set $V_D$ of all vertices of degrees at most $D$ in $G$, that is, $G_D=G[V_D]$.
Then, clearly, $\Delta(G_D)\le D$ and
$$|E(G_D)| = |E(G)|-|\{e\in E(G):\; e\cap (V\setminus V_D)\neq\emptyset\}|\ge\frac12\sum_{i=1}^NY_i-\sum_{k=D+1}^{\lfloor a/r\rfloor}kZ_k,$$
 where $Z_k=|\{i\in[N]: Y_i=k\}|$, and thus,
$$\nu(G)\ge\nu(G_D)\ge\frac{|E(G_D)|}{2D}\ge\frac1{2D}\left(\frac12\sum_{i=1}^NY_i-\sum_{k>D}kZ_k\right).$$
Hence, recalling \eqref{EY} and noticing that $\E Z_k=N\PP(Y_1=k)$, we have
$$\E(\nu(G))\ge\frac N{2D}\left(\frac1{8}-\sum_{k>D}k\PP(Y_1=k)\right).$$

It remains to estimate $\PP(Y_1=k)$ from above. Very crudely, to create a matching satisfying $Y_1=k$, one has to select $k$ other sets $A_i$, choose $r$ vertices from each of them, and match them with some $kr$ vertices of $A_1$. In the estimates below, we ignore the demand that $Y_1$ is precisely $k$, so, in  fact, we estimate from above $\PP(Y_1\ge k)$. We thus have
$$\PP(Y_1=k)\le\PP(Y_1\ge k)\le\binom Nk\binom ar^k\binom a{rk}(rk)!\frac{\alpha_{n-rk}}{\alpha_n}\le\frac{N^ka^{2rk}2^{rk}n^{rk}}{k!r!^k}\cdot\frac{(2n-2rk)!}{(2n)!}.$$
Using the inequality $1-x\ge e^{-2x}$ valid for $x\le 1/2$, the last fraction can be estimated as
$$\frac{(2n-2rk)!}{(2n)!}=\frac1{(2n)^{2rk}\left(1-\frac1{2n}\right)\cdot\ldots\cdot\left(1-\frac{2rk-1}{2n}\right)}\le\frac{e^{2r^2k^2/n}}{(2n)^{2rk}}.$$
Since $a^{2r-1}=r!n^{r-1}$ and  $k\le a/r$, we have $k^2=o(n)$ and can bound, roughly, $e^{2r^2k^2/n}\le2$.
Consequently, recalling that  $N=2n/a$ and using the bound $(k-1)!\ge ((k-1)/3)^{k-1}$, we infer that
$$k\PP(Y_1=k)\le\frac{2a^{2rk-k}}{(k-1)!r!^k(2n)^{rk-k}}=\frac{2}{(k-1)!}\le2\left(\frac{3}{k-1}\right)^{k-1}\le 2\left(\frac12\right)^{k-1}$$
for $k\ge7$.
Setting $D=6$ and summing over all $k>D$, we thus obtain the bound
$$\sum_{k>D}^{\lfloor a/r\rfloor}k\PP(Y_1=k)\le2\sum_{k>6}^{\infty}\left(\frac{1}{2}\right)^{k-1}=2\sum_{k\ge 6}^{\infty}\left(\frac12\right)^{k}=4\cdot 2^{-6}=\frac 1{16}.$$
Finally,
$$\E(\nu(G))\ge\frac N{12}\left(\frac18-\frac 1{16}\right)\ge\frac N{200}=\Theta\left(n^{{r}/{(2r-1)}}\right)$$
which, together with \eqref{concentrat}, completes the proof.
\end{proof}

\section{Final remarks}\label{fr}
Let us conclude the paper with some suggestions for future studies. The first natural goal is to extend our Erd\H{o}s--Szekers-type results to $r$-uniform ordered matchings for arbitrary $r\geqslant 4$. We have inspected more carefully the case of $r=4$ observing some similar phenomena as in the two smaller cases, $r=2$ and $r=3$. In particular, among the 35 possible relations between pairs of ordered quadruples there are exactly 27 that are collectable. This boosts some hope that by cleverly handling the remaining 8 relations, one can, indeed, generalize Theorem \ref{Theorem E-S for triples} with the exponents $1/9$ replaced by $1/27$.

Returning to the case $r=3$, we are not yet completely done, as we were unable to construct a general counterexample showing the optimality of Theorem \ref{Theorem E-S for triples} for arbitrary values of \emph{all} 9 parameters (c.f.\ Remark \ref{counter}).

\begin{problem}\label{problem1} For all positive integers $a_{XY}$, where $X,Y\in\{\mathtt{L,S,W}\}$, $(X,Y)\neq(\mathtt{L,L})$, construct a matching $M$ of size
$n=\prod_{X,Y}a_{xy}$ such that neither $M$ contains a semi-line of size 2, nor for any pair $(X,Y)$ does it contain a  sub-matching of $a_{XY}+1$ triples every two of which are in relation ${\mathcal R}_{XY}$.
\end{problem}
\noindent As mentioned earlier, setting $a_{\mathtt{LL}}=1$ is not a restriction at all, as in the general case one may simply concatenate $a_{\mathtt{LL}}$ disjoint copies of the matching described in Problem \ref{problem1}

 \medskip

A related problem is to estimate the size of unavoidable patterns  in random $r$-matchings $\rm^{(r)}_n$ with $r\geqslant 3$. For $r=2$, we have established (c.f.\ Theorem~\ref{lowerLSW}) that it is a.a.s.\ of the order $\Theta(\sqrt n)$. As it was already mentioned at the end of Section~\ref{section:random}, as a consequence of Proposition \ref{upbr}, for $r=3$ a.a.s.\ none of the nine collectable sub-matchings (as defined in Subsection~\ref{section:hyper}) has  size bigger than~$2n^{1/3}$.
It would be nice to prove a complementary lower bound. Although for 3-uniform lines this does not seem to be difficult (as, likely, Lemma~\ref{lemma:edges} can be extended for hyperedges),  showing that a.a.s.\  $\rm^{(3)}_n$ contains \emph{every} collectable sub-matching of size $\Omega(n^{1/3})$ will require some new ideas.

For arbitrary $r$, it too seems natural to expect that, as in the case  $r=2$, all  homogeneous substructures (corresponding to collectable relations) of $\rm^{(r)}_n$ should  a.a.s.\  have size  $\Theta(n^{c_r})$, for some constant $0<c_r<1$. By Theorem \ref{lowerLSW} we know that $c_2=1/2$, while guided by Proposition \ref{upbr}, we suspect that $c_r=1/r$.
\bigskip

Other open problems can be formulated for twins in ordered matchings. Based on what we proved here we state the following conjecture.

\begin{conjecture}\label{Conjecture Twins}For every fixed $r\geqslant 2$, we have $t_r^{\match}(n)=\Theta\left(n^{\frac{r}{2r-1}}\right)$.
	\end{conjecture}
The same statement is conjectured for twins in permutations (see \cite{DGR1}), and, by our results, we know that both conjectures are actually equivalent.

 One could also study the size of twins in $r$-uniform ordered matchings or, more generally, in arbitrary ordered graphs or hypergraphs. For ordinary \emph{unordered} graphs there is a  result of Lee, Loh, and Sudakov \cite{LLSud} giving an asymptotically exact answer of order $\Theta(n\log n)^{2/3}$. It would be nice to have an analogue of this result for ordered graphs.

\subsection*{Acknowledgements}

The first author was supported in part by Simons Foundation Grant \#522400. The second author was supported in part by Narodowe Centrum Nauki, grant 2020/37/B/ST1/03298. The third author was supported in part by Narodowe Centrum Nauki, grant\linebreak 2018/29/B/ST1/00426.

\providecommand{\bysame}{\leavevmode\hbox to3em{\hrulefill}\thinspace}
\providecommand{\MR}{\relax\ifhmode\unskip\space\fi MR }

\end{document}